\documentclass[a4paper,11pt]{amsart}

\usepackage[latin1]{inputenc}
\usepackage[T1]{fontenc}
\usepackage[english]{babel}

\usepackage{mathrsfs}
\usepackage{amscd}
\usepackage{amsfonts}
\usepackage{amsmath}
\usepackage{amssymb}
\usepackage{amstext}
\usepackage{amsthm}
\usepackage{amsbsy}
\usepackage{comment}

\usepackage{cleveref}
\usepackage{tikz-cd}
\usepackage{xspace}
\usepackage[all]{xy}
\usepackage{graphicx}
\usepackage{url}
\usepackage{latexsym}

\newtheorem{theo}{Theorem}[section]
\newtheorem{lem}[theo]{Lemma}

\newtheorem{cor}[theo]{Corollary}
\newtheorem{fact}[theo]{Fact}

\newtheorem{claim}[theo]{Claim}

\theoremstyle{definition}

\newtheorem{rem}[theo]{Remark}
\newtheorem{ex}[theo]{Example}

\newcommand{\N}{\ensuremath{\mathbb{N}}}
\newcommand{\Z}{\ensuremath{\mathbb{Z}}}

\newcommand{\R}{\ensuremath{\mathbb{R}}}

\newcommand{\M}{\ensuremath{\mathcal{M}}}

\newcommand{\vs}{\vspace{0.2cm}}

\DeclareMathOperator{\Aut}{Aut}

\DeclareMathOperator{\SL}{SL}
\DeclareMathOperator{\SO}{SO}

\begin{document}

\author{Annalisa Conversano, Alf Onshuus and Sacha Post}

\title{Real Lie groups and o-minimality}

\date{\today. \\
{\bf Keywords:} real Lie groups, o-minimality, definable groups}

\vspace{-.5cm}
\begin{abstract}
We characterize, up to Lie isomorphism, the real Lie groups that are definable in an o-minimal expansion of the real field. For any such group, we find a Lie-isomorphic 
group definable in $\R_{\exp}$ for which any Lie automorphism is definable.
 \end{abstract}

\maketitle

\thispagestyle{empty}

\vs

 \section{Introduction and preliminaries}

In a seminal work (\cite{Pi88}), Anand Pillay proved that any group definable in an o-minimal expansion of the real closed field $\mathbb R$ was a Lie group. In this paper we completely specify this relation by characterizing which Lie groups $G$ are Lie isomorphic to a group which is definable in an o-minimal expansion of $(\mathbb R, <, +, \cdot, 0, 1)$. Since groups definable in o-minimal structures are subject to be studied with the powerful tools of o-minimality, we believe that this characterization will prove to be quite useful.

\medskip

We will always work in o-minimal expansions of $(\mathbb R, <, +, \cdot, 0, 1)$.  The o-minimal exponential field $(\mathbb R, <, +, \cdot, 0, 1, \exp)$ is denoted by $\R_\text{exp}$. We consider real Lie groups $G$ and refer to \emph{Lie isomorphisms} simply by \emph{isomorphisms}.
We will say that $G$ is \emph{definable} when $G$ is \emph{definable in an o-minimal expansion of $(\mathbb R, <, +, \cdot, 0, 1)$}.
When we say that $G$ \emph{has a definable copy} we mean that $G$ is isomorphic to a definable group, and we will usually denote a definable copy of $G$ by $G_\text{def}$. By saying that $G$ has a \emph{semialgebraic copy}, we mean that there is a copy $G_\text{def}$ definable in $(\mathbb R, +, \cdot, 0, 1)$.

\medskip

This paper concludes and builds upon results by many researchers. Any Lie group $G$ contains a maximal connected normal solvable subgroup $\mathcal R(G)$, its solvable radical, and maximal connected semisimple subgroups $S$ (all conjugates) called the Levi subgroups. When $G$ is connected, then $G$ can be decomposed as the product of its solvable radical and any Levi subgroup, a product which is known as a \emph{Levi decomposition of $G$}. It is therefore natural, if one is to study the existence of definable copies of a Lie group, to begin by understanding when semisimple and solvable Lie groups admit definable copies.

\medskip

It is well known that if $G$ is definable then so is its solvable radical $\mathcal R(G)$. In  \cite[Theorem 4.5]{PPSIII}  the authors proved that if $G$ is a definable linear group then Levi subgroups are definable and by \cite[Theorem 1.1]{CP-Levi} in general Levi subgroups of definable groups are a countable union of definable sets (see Example \ref{levi-non-normal}).

\medskip

By Theorem 4.3 in \cite{PPSIII} any semisimple matrix group is semialgebraic and by \cite[Corollary 3.3]{OPP96} the quotient of a definable group by its center has a faithful representation. This line of research was extended in \cite{HPP} where the authors established criteria for when a central extension of a definable semisimple group was definable. We combine these results in Theorem \ref{centralextension} showing that a central extension $H$ of a connected semisimple Lie group has a definable copy if and only if both $H$ and its center $\mathcal Z(H)$ have finitely many connected components. Moreover, when it exists, such a definable copy can be found to be semialgebraic. In Section \ref{sec:automorphisms} we show that any automorphism of the connected component of such semialgebraic copy is semialgebraic as well.

\medskip

On the solvable side of the problem, any definable group $G$ contains a maximal normal definable torsion-free subgroup $\mathcal{N}(G)$ (\cite[Proposition 2.6]{CP1}) that is definably connected and solvable (\cite[Corollary 2.4 \& Claim 2.11]{PeSt08}), so $\mathcal{N}(G) \subseteq \mathcal R(G)$. The quotient $\mathcal{R}(G)/\mathcal{N}(G)$ is compact and abelian (see \cite[Theorem 3.1]{survey}, \cite[Theorem 5.12]{Ed} and \cite[Corollary 5.4]{Peterzil-Starchenko1}).  Since compact groups admit a faithful representation whose image is algebraic, the solvable radical of any group with a definable copy must be isomorphic to an extension of a compact linear algebraic group by a torsion-free group with a definable copy. This characterization of the solvable case was completed in \cite{COS}, where it is shown that $\mathcal{N}(G)$ is supersolvable and that a solvable connected Lie group has a definable copy if and only if it is triangular-by-compact (see Fact \ref{Lie-done-solvable}).

\medskip

In \cite{OP19}, the second and third author showed that a connected Lie group whose Levi subgroups have finite center has a definable copy if and only if its solvable radical has a definable copy (see Fact \ref{Lie-done-simple}).

\medskip

In Theorem \ref{general} we generalize the previous results and show that a Lie group $G$ has a definable copy if and only if:
\begin{itemize}
\item $G$ has finitely many connected components,
\item the center $\mathcal Z(G)$ of $G$ has finitely many connected components, and
\item there is a normal simply-connected subgroup $N$ which admits a triangular representation and such that $\mathcal R(G)/N$ is compact.
\end{itemize}

\medskip

The sketch of the proof is as follows. We first prove the connected case in Theorem \ref{connected}, by showing that any connected group $G$ satisfying the above conditions can be realized as a quotient of a semidirect product of a definable solvable group and a definable central extension of a linear semisimple group. The former is definable in $\mathbb R_{\text{exp}}$ and we will prove the latter has a semialgebraic copy (Theorem \ref{centralextension}), which then allows us to build the definable copy of $G$ in $\mathbb R_{\text{exp}}$.

The generalization from the connected case to the general case is mainly about understanding finite order automorphisms of definably connected groups. In Theorem \ref{definable automorphisms} we prove that any connected Lie group that has a definable copy is isomorphic to a definable group \emph{for which any Lie group automorphism is definable}. The general case follows.

\subsection{Notation}

Let $G$ be either a real Lie group or a group definable in an o-minimal expansion of a real closed field. If $G$ is definable, the connected component $G^0$ of the identity is definable and the notions of connectedness and definable connectedness coincide. As above, we will denote by $\mathcal R(G)$ the solvable radical of $G$ (that is, the largest solvable normal --definable-- connected subgroup of $G$), and by $\mathcal Z(G)$ the center of $G$.

If $G$ is connected, then a \emph{Levi decomposition} of $G$ is a decomposition of the form $G = RS$ where $S$ is a maximal connected semisimple group, called a \emph{Levi subgroup} of $G$. We know the following:

\begin{itemize}
\item $R\cap S$ is discrete.

\item If $G$ is a Lie group, then $S$ is a maximal connected semisimple subgroup of $G$, and unique up to conjugation (\cite[Theorem 3.18.13]{Varadarajan}).

\item If $G$ is a definable group, then $S$ is a maximal connected ind-definable subgroup of $G$, and unique up to conjugation (\cite[Theorem 1.1]{CP-Levi}).
\end{itemize}

We denote by $G^0$ the (definably) connected component of $G$. If $G$ is a definable group, $G/G^0$ is finite. If $G$ is a Lie group, $G/G^0$ is a discrete (possibly infinite) group.

\medskip

We will say that a Lie group $G$ is \emph{linear} if it has a continuous and faithful  finite dimensional representation $\rho$ and we will call the image $\rho (G)$ a \emph{matrix group}.

\medskip


We want to emphasize that when we use the reduced terminology ``definable'' in, say a theorem or a proof, the objects are always definable in the same o-minimal structure through the entire proof or statement.

\medskip

 \begin{fact} \label{Lie-done-solvable}
 Let $G$ be a solvable connected real Lie group. Then the following are equivalent:
 \begin{enumerate}
 	\item $G$ is isomorphic to a semidirect product $N \rtimes K$ where $N$ is a closed subgroup of $T^+_n(\R)$ the subgroup of upper triangular matrices with positive elements on the diagonal for some $n \in \N$ and $K$ is isomorphic to $\SO_2(\R)^k$, for some $k \in \N$.
 	\item $G$ is triangular-by-compact.
 	\item $G$ has a definable copy.
 	\item $G$ has a definable copy as in $(1)$ definable in $\mathbb{R}_{\exp}$.
 \end{enumerate}
\end{fact}

\begin{proof}
The equivalence of the first three is Theorem 5.4 in \cite{COS}, where by \emph{triangular-by-compact} we mean that $G$ has a simply-connected
triangular closed subgroup $N$ such that $G/N$ is compact.

 (4) implies (3) is clear.

Now, given $N$ and $K$ as in (1) with $N$ an upper triangular subgroup of a matrix group and $K=\SO_2(\R)^k$ for some $k$. Any action from $K$ to $N$ induces an action from $K$ to the Lie algebra $\mathfrak n$ of $N$ which by construction is a vector subspace of $M_n(\mathbb R)$. Therefore, $Aut(\mathfrak n)\subseteq GL_{n^2}(\mathbb R)$ is a matrix group and the graph of the action of $K$ on $\mathfrak n$ is a subgroup of the matrix space $\SO_2(\R)^k\times GL_{n^2}(\mathbb R)$. This graph is a matrix group isomorphic to $K$ hence it is algebraic (see \cite[Theorem 2.6.4]{Abbaspour}). The exponential map from $\mathfrak n$ to $N$ is definable in $\mathbb{R}_{\exp}$ (because $\mathfrak n$ is triangular by assumption), so $N \rtimes K$ is definable in $\mathbb{R}_{\exp}$, and (4) holds.
\end{proof}

\medskip

\begin{fact}\label{Lie-done-simple}
Suppose $G$ is a connected semisimple Lie group. Then $G$ has a definable copy if and only if it has a semialgebraic copy and if and only if its center is finite.
\end{fact}

\begin{proof}
If $G$ is centerless, then it is semialgebraic by \cite{PPSIII}. The finite center case follows from Theorem 6 in \cite{OP19}.
\end{proof}

\medskip

Finally, the following is Theorem 3 in \cite{OP19}\footnote{This result appears in the Ph.D. thesis of the third author.}.

\begin{fact}\label{Lie-done-OP} Suppose $G$ is a connected Lie group whose Levi subgroups have finite center.\footnote{This is always the case for linear groups.} Then $G$ has a definable copy if and only if it has a definable copy in $\mathbb{R}_{\exp}$ if and only if its solvable radical has a definable copy.
Moreover, if $G = RS$ is a Levi decomposition of $G$, a definable copy $G_\text{def}$ can be obtained as a quotient of a definable semidirect product $R_\text{def} \rtimes_\varphi  S_\text{def}$ of definable copies of $R$ and $S$ by a finite subgroup isomorphic to $R \cap S$. The definable action $\varphi \colon  S_\text{def} \times R_\text{def} \to R_\text{def}$ mimicks the conjugation action of $S$ on $R$.
\end{fact}

Using Fact \ref{Lie-done-solvable}, the complete characterization of Lie groups that have a definable copy will be achieved by the following theorem:

 \begin{theo} \label{main}
A real Lie group $G$ has a definable copy if and only if both $G$ and its center have finitely many connected components and the solvable radical of $G$ is triangular-by-compact. Moreover, a  copy exists in $\mathbb{R}_{\exp}$.
 \end{theo}

\bigskip

\begin{rem} \label{finite-intersection}
  Let $G = RS$ be a Levi decomposition of a connected real Lie group $G$. Suppose $R \cap S$ is finite. Then
 \[
  G \mbox{ has a definable copy }  \Longleftrightarrow\  R \mbox{ and } S \mbox{ have a definable copy}
  \]
  That is, $G$ has a definable copy if and only if $R$ is triangular by compact  and $Z(S)$ is finite.
\end{rem}

\begin{proof}
$(\Rightarrow)$ If $S$ does not have a definable copy, then $\mathcal Z(S)$ is infinite, by  \cite[Theo 3]{OP19}.
Therefore $S/(R \cap S)$ has infinite center too, since $R \cap S$ is finite. It follows that $S/(R \cap S) \cong G/R$ does not have a definable copy, contradiction. \\
\noindent
 $(\Leftarrow)$ If $S$ has a definable copy, then $\mathcal Z(S)$ is finite, and $G$ has a definable copy by Fact \ref{Lie-done-OP}.
\end{proof}

So we need to concentrate in the case where $R \cap S$ is infinite. Since $R \cap S$ is a discrete normal subgroup of $S$, it is central in $S$ and thus in this case $S$ cannot be definable, as it has an infinite discrete center. In all the examples from the literature of definable groups with no definable Levi subgroups \cite[Ex 5.7-5.9]{survey}, there is a unique Levi subgroup (that is, the Levi subgroup is normal in the group), so one may wonder if this is always the case. Below is an example where Levi subgroups are not normal:

 \begin{ex} \label{levi-non-normal}
 Denoted by $\pi \colon  \widetilde{\SL}_2(\R) \to  \SL_2(\R)$ the universal covering map of $\SL_2(\R)$, let $s \colon  \SL_2(\R) \to \widetilde{\SL}_2(\R)$ be a section for it. Then the 2-cocycle $h_s \colon \SL_2(\R) \times \SL_2(\R) \to \mathcal Z(\widetilde{\SL}_2(\R))$ given by $h_s(A, B) = s(A)s(B)s(AB)^{-1}$ is a definable map with finite image by \cite[Theo 8.5]{HPP} (see section 2).

 On the set $G = \R \times \R^2 \times \SL_2(\R)$ consider the group operation given by
 \[
 (t, \bar{x}, A) \ast (s, \bar{y}, B) = (t + s + h_s(A, B), A\bar{y} + \bar{x}, AB)
 \]

 \noindent
 where $\mathcal Z(\widetilde{\SL}_2(\R))$ has been identified with $(\Z, +)$.  The definable group $(G, \ast)$ is a central extension of $\R^2 \rtimes \SL_2(\R)$ -- where $\SL_2(\R)$ acts on $\R^2$ by matrix multiplication -- by $(\R, +)$. A Levi subgroup is $S = \Z \times \{\bar{0}\} \times \SL_2(\R)$, and is isomorphic to $\widetilde{\SL}_2(\R)$ by construction. Since $\SL_2(\R)$ is not normal in $G/\mathcal Z(G)$, then $S$ is not normal in $G$.
 \end{ex}

\begin{fact}\label{Levi-decomposition}
Let $G$ be a connected definable group and $Z= \mathcal Z(G)$ the center of $G$. Then $G/Z$ and $G/Z^0$ have definable Levi subgroups.
\end{fact}

\begin{proof}
By \cite[Cor 3.3]{OPP96}, $G/Z$ is a matrix group, and by \cite[Theo 4.5]{PPSIII} matrix definable groups have a definable Levi decomposition. Since $Z$ is a finite extension of $Z^0$, it follows that $G/Z^0$ has definable Levi subgroups too.
\end{proof}

\begin{cor} \label{ZSdef}
 Let $G$ be a connected real Lie group with center $Z$ and $S$ a Levi subgroup. If $G$ is definable then $ZS$ and $Z^0 S$ are definable subgroups of $G$.
\end{cor}

 \begin{proof}
 If $G$ is definable, then  $ZS$ and $Z^0 S$ are pre-images in $G$ of definable Levi subgroups in $G/Z$ and $G/Z^0$, respectively.
\end{proof}

\vs

We will characterize when a connected real Lie group $G$ is isomorphic to a group definable in an o-minimal expansion of the real field as follows. Let $R$ be the solvable radical of $G$, $S$ a Levi subgroup of $G$ and $Z$ the center. We will show that whenever $Z$ has a definable copy, there is a semialgebraic copy of $ZS$ (Section \ref{ZS}) so that if $R$ also has a definable copy (which can be defined in $\mathbb R_\text{exp}$ by Fact \ref{Lie-done-solvable}) we can build a definable (in $\mathbb R_\text{exp}$) copy of $G$ as a quotient of a definable semidirect product of definable copies of $R$ and $ZS$ mod out by a definable group isomorphic to $R\cap ZS$ (Section \ref{all together now}).

\medskip

In the proof of Theorem \ref{general} we will reduce to the connected case, and then generalise using the following lemma, which is probably well known but we could not find a reference for it.

\begin{lem} \label{central-finite}
Let $G$ be a group and $H$ a normal subgroup of finite index. Assume that $H$ is definable in a structure $\mathcal M$ and that for a set $\mathcal G:=\{g_\sigma\}_{\sigma\in G/H}$ of representatives of the $H$-coclasses in $G$ there are definable (in $\mathcal M$) maps  $f_\sigma:H\rightarrow H$ such that $f_\sigma(h)=g_\sigma h g_\sigma^{-1}$.

Then $G$ is isomorphic to group definable in $\mathcal M$.
\end{lem}

\begin{proof}
Let $g_\sigma, g_\mu$ be any two elements of $\mathcal G$, and let $g_{\sigma\mu}\in \mathcal G$ and $h_{\sigma\mu}\in H$ be such that $g_\sigma g_\mu=g_{\sigma\mu}h_{\sigma\mu}$. Let $x,y\in H$.

\begin{eqnarray*}
g_\sigma x g_\mu y & = g_{\sigma\mu} g_{\sigma\mu}^{-1}(g_\sigma x g_\sigma^{-1}) (g_\sigma g_\mu) y  \\
 & = g_{\sigma\mu} \left(g_{\sigma\mu}^{-1} f_\sigma\left(x\right) g_{\sigma\mu}\right) h_{\sigma\mu} y \\
 &= g_{\sigma\mu} f_{\sigma\mu}^{-1}\left(f_{\sigma}\left(x\right)\right) h_{\sigma\mu} y.
\end{eqnarray*}

Therefore, if we define $G_\text{def}$ as $\mathcal G\times H$ with multiplication in $G_\text{def}$ defined by
\[
(g_\sigma, x)\odot (g_\mu, y):= \left(g_{\sigma\mu}, f_{\sigma\mu}^{-1}\left(f_{\sigma\mu}\left(x\right)\right) h_{\sigma\mu} y\right),
\]
then $G_\text{def}$ is a definable group and the map $G_\text{def}\rightarrow G$ defined by $(g_\sigma, x)\mapsto g_\sigma x$ is an isomorphism.
\end{proof}

%
%

\begin{rem}
If $G$ is a Lie group and $\M$ is an o-minimal structure over the real numbers, then Lemma \ref{central-finite} provides a Lie isomorphism with a definable Lie group, since any group definable in an o-minimal structure with domain $\R$ is a Lie group by \cite{Pi88} and any bijective homomorphism of Lie groups is a Lie isomorphism (see \cite[Theorem 3.4 pg.18]{OnishchikVinbergI}).
\end{rem}


\section{Central extensions of semisimple groups}\label{ZS}

In this section we focus on Lie groups $G$ that are central extensions of a connected semisimple group. We show that  $G$ has a semialgebraic copy if and only if $G$ and its center have finitely many connected components. The following fact allows us to find a direct complement of $G^0$, whenever $G$ is abelian:

\begin{fact} \label{abelian-divisible}\cite[5.2.2]{Scott}
Let $H < G$ be abelian groups. If $H$ is divisible, then $G = H \times K$ for some $K < G$.
\end{fact}

Central extensions of definable groups have been previously studied by Hrushovski, Peterzil and Pillay in \cite{HPP}. As they recall, given a central group extension $1 \to A \to H \xrightarrow{\pi} G$, if $s \colon G \to H$ is a section for $\pi$, and $h_s(x, y) = s(x)s(y)s(xy)^{-1}$, then $h_s$ is a $2$-cocycle from $G \times G$ to $A$ and the group $H$ is isomorphic to the group $H'$ whose underlying set is $A \times G$ and whose group operation is given by $(t, x) \cdot (s, y) = (t + s + h_s(x, y), xy)$. In this context, assume that $G$ be a definably connected group definable in an o-minimal expansion \M\ of the real field (therefore $G$ is a connected Lie group) and suppose $\pi \colon H \to G$ is a covering homomorphism with kernel $\Gamma$, for some connected Lie group $H$.
 Given an injective homomorphism $f \colon \Gamma \to A$ into an abelian group $A$, one can form the group $H_A = H \times_{\Gamma} A$, given by the amalgamated direct product of $H$ and $A$, where isomorphic subgroups $\Gamma < H$ and $f(\Gamma) < A$ have been identified. The group
 $H_A$ is therefore isomorphic to a group whose underlying set is $A \times G$ and whose group operation is determined by a
 $2$-cocycle $h_s$ as above. 

We will use the following Fact from \cite{HPP}. In the setting described in the previous paragraph, the following holds:

\begin{fact} \cite[Theo 8.5]{HPP} \label{HPPcocycle}
 The $2$-cocycle $h_s \colon G \times G \to \Gamma$ induced by a section $s$   of the cover $\pi \colon H \to G$ is definable in \M\ with finite image in $\Gamma$. It follows that the group $H_A = H \times_{\Gamma} A$ is definable in the two-sorted structure consisting of \M\ and $(A, +)$.
\end{fact}

\begin{theo}\label{centralextension}
Let $G$ be a real Lie group. Suppose $G$ is a central extension of a (possibly trivial) connected semisimple group. Then $G$ has a definable copy if and only if both $G$ and its center have finitely many connected components. If this condition holds, then $G$ admits a semialgebraic copy.
\end{theo}

 \begin{proof}
 Let $Z =\mathcal{Z}(G)$ be the center of $G$. If $G$ is definable, then $G^0$, $Z$ and $Z^0$ are definable too. Since $G/G^0$ and $Z/Z^0$ are discrete groups, they need to be finite.

\bigskip
Conversely, suppose both $G$ and $Z$ have finitely many connected components.
We will define $G_\text{def}$ as an extension of a semisimple connected matrix group $S$ by an abelian group $A = F \times \mathbb{R}^n \times\SO_2(\R)^k$ for some $n$, $k \in \mathbb{N}$ and finite $F$. The underlying set of $G_\text{def}$ will be $A\times S$ and the group operation given in terms of a definable $2$-cocycle.

\bigskip
Assume first $G$ is connected.
It is well known that any abelian connected Lie group is  isomorphic to $\R^n \times \SO_2(\R)^k$ for some $n, k \in \N$. This completes the case when the semisimple group is trivial.

\bigskip

Suppose now $G$ is a connected central extension of an infinite connected semisimple group.  If $Z$ is finite, then $G$ is semisimple and has a definable copy by Fact \ref{Lie-done-simple}.

Assume $\dim Z > 0$. Note that the connected component of the center $Z^0$ is the solvable radical $R$ of $G$, since otherwise $R/Z^0$ would be an infinite connected solvable normal subgroup of a semisimple group, contradiction.

Since connected semisimple Lie groups are perfect, it follows that the commutator subgroup $G'$ is the unique Levi subgroup of $G$ and $G = ZG'$ is the unique Levi decomposition of $G$. Note that $\mathcal{Z}(G') = G' \cap Z$, so the quotient group $G/Z \cong G'/(G' \cap Z) = G'/\mathcal{Z}(G')$ is a centerless semisimple group, and therefore has a semialgebraic copy $S$.

Hence $G$ is Lie isomorphic to a group $H = A \times _{\Gamma} H'$, where $A = \mathcal{Z}(H) = F \times \R^n \times \SO_2(\R)^k$ (for some finite $F$ and $n, k \in \N$ by Fact \ref{abelian-divisible}), $S = H'/Z(H')$ is a  semialgebraic group and $\Gamma = A \cap H' = Z(H')$. As $(A, +)$ is semialgebraic, then $H = A \times _{\Gamma} H'$ is semialgebraic by Fact \ref{HPPcocycle}.

That is, given $s \colon S \to H'$ a section for the canonical projection $\pi \colon H' \to S$  and  $h_s \colon S \times S \to \Gamma$ be the 2-cocycle induced by $s$, $h_s$ is a semialgebraic map and has finite image. So we can
consider on $H_{\text{def}}:= A \times S$ the group operation given by
  \[
  (t, x) \odot (s, y) = (t + s + h_s(x, y), xy)
  \]

 As observed before, the resulting semialgebraic group $(H_\text{def}, \odot)$ is isomorphic to $H$ and therefore to $G$.

\bigskip

Suppose now $G$ is not connected and let $\pi \colon G \to  G/Z$ be the canonical projection. As $G/Z$ is connected semisimple, it follows that $\pi(S_1) = G/Z$ for any Levi subgroup $S_1$ of $G^0$. Therefore $G = ZS_1$ and, as for connected groups, there is some Lie group $H = A \times_{\Gamma} H'$ ($\Gamma = A \cap H' = Z(H')$) isomorphic to $G$ such that $A$ and $S = H'/Z(H')$ are semialgebraic. So one can buid a semialgebraic copy $H_\text{def}$ with underlying set $A \times S$ and group operation given by the same construction as in the connected case.
\end{proof}

\section{The connected case} \label{all together now}

In this section we give a characterization of connected Lie groups with a definable copy and we show that when it exists, such a definable copy can be found in $\R_{\exp}$. The following fact holds for both Lie and definable groups.

\begin{fact} \label{semisimple-centerless}
Let $G$ be a connected centerless semisimple group. Then the only normal solvable subgroup of $G$ is the trivial one.
\end{fact}

\begin{proof}
 Suppose $N$ is a solvable normal subgroup of $G$. We can prove that $N$ is trivial by induction on $n = \dim N$. If $n = 0$, then $N$ is discrete, therefore it is central, and it must be trivial. If $n > 0$, then $N^0$ is also solvable and normal in $G$. We know that $N^0$ cannot be abelian, because $G$ is semisimple. So the derived subgroup of $N^0$ is a non-trivial connected solvable normal subgroup of $G$ with dimension less than $n$, and it must be trivial by induction hypothesis. But then $N^0$ is abelian, contradiction.
\end{proof}

\begin{lem}\label{lemma:intersection}
Let $G = RS$ be a Levi decomposition of a connected real Lie group $G$.
Set $H = ZS$, where $Z$ is the center of $G$. Then $H \cap R = D \times Z^0$ for some discrete group $D$ central in $H$. In particular, $H \cap R$ is abelian.
\end{lem}

\begin{proof}
Note that $\mathcal{Z}(H) = Z \cdot \mathcal{Z}(S)$, so $H/\mathcal{Z}(H) \cong S/\mathcal{Z}(S)$ is a centerless semisimple group, since $S$ is a connected semisimple group. Because $H \cap R$ is a solvable normal subgroup of $H$, then $H \cap R \subseteq \mathcal{Z}(H)$ by Fact \ref{semisimple-centerless}.
\end{proof}

\begin{theo} \label{connected}
Let $G$ be a connected real Lie group. Then $G$ has a definable copy if and only if its center has finitely many connected components and its solvable radical has a definable copy. If this holds, $G$ has a definable copy in $\mathbb{R}_{\exp}$.
\end{theo}

\begin{proof}
The conditions are necessary because  the center of any definable group is definable (and therefore it has finitely many connected components) and so is its solvable radical.

\bigskip

Conversely, let $G$ be a connected Lie group with center $Z$ and solvable radical $R$ satisfying the hypothesis. Suppose $G = RS$ is a Levi decomposition of $G$. We will build a definable copy of $G$ as a quotient of a semidirect product of definable copies of $ZS$ and $R$. Set $H = ZS$ and $\overline{S} = H/Z \cong S/(S \cap Z)$. This construction will be done in $\mathbb{R}_{\exp}$, so for the rest of the proof unless we explicitly say otherwise, by \emph{definable} we mean \emph{definable in $\mathbb{R}_{\exp}$}.

\bigskip

 Because $R$ is a normal subgroup in $G$, the subgroup $H$ acts on $R$ by conjugation. That is, there is a homomorphism $\gamma \colon H \to \Aut(R)$ such that $\gamma(x)(r) = x^{-1}rx$ for each $x \in H$ and each $r \in R$.  Since $Z \subset \ker(\gamma)$, there is a homomorphism $\overline{\gamma} \colon \bar{S} \to \Aut(R)$ such that the following diagram commutes:
\[
\begin{tikzcd}[column sep=large, row sep=large]
H  \ar[d, "\pi"' ]   \ar[r,   "\gamma"] & \text{Aut}(R)\\
\bar{S} = H/Z \ar[ur,   "\overline{\gamma}"']
\end{tikzcd}
\]

 \noindent
That is, given $\bar{x} \in \overline{S}$ and $r \in R$, $\overline{\gamma}(\bar{x})(r) = x^{-1}rx$, for any $x \in H$ with $x \in \bar{x}$. \\

\begin{claim}\label{cl:3.4}
There are definable copies  $R_\text{def}$ and  $\overline{S}_\text{def}$,  Lie isomorphisms  \\ $f \colon \overline{S}_\text{def} \to \overline{S}$ and $g \colon R_\text{def} \to R$, and a definable action  $\varphi \colon \overline{S}_\text{def}\times R_\text{def} \to R_\text{def}$ mimicking $\overline{\gamma}$. That is, for any $x\in \overline{S}_\text{def}$ and $r_1,r_2\in R_\text{def}$,
\[
\varphi (x, r_1) = r_2 \quad \text{ if and only if }
\quad \overline \gamma (f^{}(x))(g(r_1)) = g(r_2)
\]

\noindent
Moreover, $R_\text{def}$ contains  definable copies $Z^0_\text{def}$ and $(Z\cap R) _\text{def}$ as definable subgroups.
\end{claim}

 \begin{proof}
 Consider the group $\overline{G} = R \rtimes_{\overline{\gamma}}\overline{S}$. It is a connected group with solvable radical $R$ and Levi subgroup $\overline{S}$. Note that because $\overline{S}$ is a subgroup of $G/Z$, it is a linear group, hence has finite center.
Therefore, by Fact \ref{Lie-done-OP} there is a definable copy  $\overline{G}_\text{def} = R_\text{def} \rtimes_\varphi   \overline{S}_\text{def}$,
Lie isomorphisms  $f \colon \overline{S}_\text{def} \to \overline{S}$ and $g \colon R_\text{def} \to R$ and a definable action $\varphi \colon \overline{S}_\text{def}\times R_\text{def} \to R_\text{def}$ as required.

\bigskip

Note that both $(Z \cap R)_\text{def} = \{ r \in \mathcal{Z}(R_\text{def}) : \varphi(\bar{x}, r) = r \ \forall \bar{x} \in \overline{S}_\text{def}\}$ and $(Z \cap R)_\text{def}^0 = Z^0_\text{def}$ are definable subgroups of $R_\text{def}$.
\end{proof}

 Set $A = Z^0_\text{def} < R_\text{def}$. Our next step is to build a definable copy $H_\text{def}$ of $H$ having $A$ as its solvable radical and a definable action $H_\text{def}\times R_\text{def} \to R_\text{def}$ mimicking the conjugation map of $H$ on $R$:

 \begin{claim}\label{cl:3.5}
There is a definable copy $H_\text{def}$ with $\mathcal{Z}(H_\text{def})^0 = A$, a Lie isomorphism $h \colon H_\text{def} \to H$ and a definable action
$\psi \colon H_\text{def}\times R_\text{def} \to R_\text{def}$ such that
\[
\psi(x, r_1) = r_2 \quad \text{ if and only if } \quad h(x)^{-1} g(r_1) h(x) = g(r_2)
\]

\noindent
where $g \colon R_\text{def} \to R$ is the isomorphism from Claim \ref{cl:3.4}.

As before, definability of every set and map is meant in $\mathbb{R}_{\exp}$.
\end{claim}

\begin{proof}
Since $H$ is a central extension of a semisimple group, by Theorem \ref{centralextension} $H$ has a definable (semialgebraic) copy if and only if $H$ and its center have finitely many connected components. Note that the subgroup $Z^0S$ is contained in $H^0$, because both $Z^0$ and $S$ are connected subgroups. Moreover $Z^0 S$ has finite index in $ZS$, as $Z/Z^0$ is finite by assumption. Therefore $H^0 = Z^0 S$ and $H$ has finitely many connected components. Notice that $G/Z$ is linear since $Z$ is the kernel of the adjoint representation of $G$, and $H/Z \cong S/(S \cap Z)$ is a Levi subgroup of $G/Z$, since $S$ is a Levi subgroup of $G$. Now, Levi subgroups of a linear groups are linear as well and have a finite center (see \cite[Chap. 18 Proposition 4.1 and Theorem 4.2]{Hochschild}). Therefore $\mathcal{Z}(H)$ is a finite extension of $Z$, which is a finite extension of $Z^0$. It follows that $\mathcal{Z}(H)$ has finitely many connected components too and $\mathcal{Z}(H)^0 = Z^0$.  \\

As done in the proof of  Theorem \ref{centralextension},  we can get a definable copy $H_\text{def}$ where the underline set of the group is $\mathcal{Z}(H)_\text{def} \times (H/\mathcal{Z}(H))_\text{def} $. By Fact \ref{abelian-divisible}, $\mathcal{Z}(H) = F_1 \times Z^0$ for some finite $F_1 < H$. Thus we can take $\mathcal{Z}(H) _\text{def}  = F_1 \times A$ and   $ (H/\mathcal{Z}(H))_\text{def} = \overline{S}_\text{def}/\mathcal{Z} (\overline{S}_\text{def})$.  \\

Note that $\mathcal{Z}(H) _\text{def} = \mathcal{Z}(H _\text{def}) $ and $(H/\mathcal{Z}(H))_\text{def}  = H_\text{def}/\mathcal{Z}(H_\text{def})$. Moreover, $\mathcal{Z}(H) = Z \cdot \mathcal{Z}(S)$, so we can take  $Z_\text{def} = F_2 \times A$ (for some finite subgroup $F_2$ of $F_1$) to be a definable subgroup of $H_\text{def}$ and $H _\text{def}/Z_\text{def} = \overline{S}_\text{def}$ by construction. Denoted by $\pi \colon H_\text{def} \to H_\text{def}/Z_\text{def}$ the canonical projection, the map $\psi \colon H_\text{def}\times R_\text{def} \to R_\text{def}$ defined by
$\psi(x, r) = \varphi(\pi(x),r)$ is the required definable action.
 \end{proof}
Now in order to build a definable copy of $G$ out of the definable copies $R_\text{def}$ and $H_\text{def}$ from before, we need to find a definable isomorphism between definable copies $(R \cap H)_\text{def}$ in them. This is provided by the following:

\begin{claim} \label{kernel}
 $H_\text{def}$ and $R_\text{def}$ contain definable subgroups $X$ and $Y$ which are copies of $H \cap R$, and such that $X=A \times F_X$ and $Y=A\times F_Y$  where $F_X$ and $F_Y$ are both isomorphic to some finite abelian group $F$. In particular, there is a definable isomorphism  $\alpha \colon X \to Y$.
 \end{claim}

\begin{proof}
By Lemma \ref{lemma:intersection} and the fact that $H$ has a definable copy, $H \cap R = F \times Z^0$, for some finite abelian $F < H \cap R$. Then $X = h^{-1}(F) \times A$ and $Y = g^{-1}(F) \times A$ are definable copies (and definable subgroups) of $H \cap R$
in $H_\text{def}$ and $R_\text{def}$ respectively. Clearly, define $\alpha \colon X \to Y$ to be $\alpha(y, a) = (g^{-1}(h(y)), a)$ for each $y \in h^{-1}(F)$ and $a \in A$.
\end{proof}
We can now put all together. Let $R_\text{def}$ be as in Claim \ref{cl:3.4}, and $H_\text{def}$ and $\psi$ as in the conclusion of Claim \ref{cl:3.5}. Take the definable semi-direct product $H_\text{def} \ltimes_{\psi} R_\text{def}$ and the map
\[\Phi \colon H_\text{def} \ltimes_{\psi} R_\text{def} \to G\] given by $\Phi (x, r) =  h(x)g(r)$ where $h \colon H_\text{def} \to H$ and $g \colon R_\text{def} \to R$ are the isomorphisms from Claims \ref{cl:3.4} and \ref{cl:3.5}. Note that $\Phi$ is a surjective smooth homomorphism and its kernel is  the set of $(x, r) \in H_\text{def} \times  R_\text{def}$ such that $h(x)g(r) = e$. In particular, the images of $x$ and $r$ belong to $H \cap R$, so $x \in X$ and $r \in Y$ and, if $\alpha \colon X \to Y$ is the definable map in Claim \ref{kernel}, then
\[
\ker \Phi = \{(x, \alpha(x)^{-1} ): x \in X \}
\]
is definable.

\medskip

$(H_\text{def} \ltimes_{\psi} R_\text{def})/ \ker \Phi$ is therefore a definable copy of $G$, as required.
\end{proof}

\begin{rem}\label{the right form}
Any connected Lie group satisfying the conditions of Theorem \ref{connected} will be isomorphic to a group
\[
G_{\text{def}}:=(H_\text{def} \ltimes_{\psi} R_\text{def})/ \{(x, \alpha(x)^{-1} ): x \in X \}\]

definable in $\R_\text{exp}$, such that the following hold:

\begin{enumerate}
\item The center $Z_{\text{def}}$ of $G_{\text{def}}$ is a linear algebraic abelian group.

\item $H_\text{def}$ is a semialgebraic group with underlying set $A\times \overline{S}$ where $A=\mathcal Z(H_\text{def}) = F \times \R^n \times \SO_2(\R)^k$ and $\overline{S}$ is a linear semisimple group isomorphic to $H_\text{def}/\mathcal Z(H_\text{def})$.

\item $R_\text{def}=K \ltimes N$ where $K$ is an matrix algebraic compact group and $N$ is a closed subgroup of upper triangular matrices,

\item $\alpha$ is a definable isomorphism between definable copies $(H \cap R)_\text{def}$ in $H_\text{def}$ and $R_\text{def}$.

\item $H_\text{def}$ is definably isomorphic to $\mathcal{Z}(G_\text{def})S$ where $S$ is a Levi subgroup of $G_\text{def}$  [$\mathcal{Z}(G_\text{def})S$ is definable by Corollary \ref{ZSdef}].

\item The solvable radical of $G_\text{def}$ is definably isomorphic to $R_\text{def}$.
\end{enumerate}
\end{rem}

\section{The general case}\label{sec:automorphisms}

In this section we will prove the following.

\begin{theo} \label{general}
Let $G$ be a real Lie group. Then $G$ has a definable copy if and only if $G^0$ has a definable copy and $G$ has finitely many connected components. This is equivalent to $G$ having a definable copy in $\mathbb{R}_{\exp}$.
\end{theo}

Using Lemma \ref{central-finite} the theorem will follow from the following, which is interesting in its own.

\begin{theo}\label{definable automorphisms}
Let $G$ be a connected Lie group whose center and solvable radical have a definable copy. Then $G$ has a definable copy $G_\text{def}$ such that every Lie automorphism of $G_\text{def}$ is definable in $\mathbb R_{\exp}$.
\end{theo}

\begin{proof}
By Theorem \ref{connected}, $G$ has a definable copy $G_\text{def}$.
Let $G_\text{def}$ be as described in Remark \ref{the right form}, and let $\sigma$ be a Lie automorphism of $G_\text{def}$.

\begin{claim}\label{moving R}
The restriction $\sigma_R$ of $\sigma$ to $\mathcal{R}(G_\text{def})$ (see Remark \ref{the right form}) is definable in $\mathbb{R}_{\exp}$.
\end{claim}

\begin{proof}
By Remark \ref{the right form}, $\mathcal R(G_\text{def})$ is definably isomorphic to $R_\text{def}$ which is a semidirect product of a compact matrix algebraic subgroup $K$ and a subgroup $N$ which is a definable closed subgroup of an upper triangular matrix group. Because the exponential from $\mathfrak n$ to $N$ is a  diffeomorphism definable in $\mathbb R _{\exp}$, the action from $K$ to $N$ can be defined
from the action from $K$ to the Lie algebra $\mathfrak n$ (which will be a subspace of a matrix vector space) of $N$. Since $\mathfrak n$ is a definable vector space, the automorphism group of $\mathfrak n$ is linear, the graph of the action is a matrix group isomorphic to $K$, so it is compact and thus algebraic. The claim follows.
\end{proof}

\medskip

\begin{claim}\label{automorphism of H}
Any automorphism of $H_\text{def}$ is definable.
\end{claim}

\begin{proof}
Let $\sigma$ be any automorphism of $H_\text{def}$. By hypothesis, $H_\text{def}$ is a definable extension of the linear semisimple group $\overline S$ by $A=\mathcal Z(H_\text{def})=F \times \R^n \times \SO_2(\R)^k$, its universe is $A\times \overline S$ and the group operation given by a definable 2-cocycle $F$. We will assume that $F(x,e_{S})=F(e_{S},x)=e_A$. Let $\sigma_A$ is the restriction of $\sigma$ to $A$, $\tau$ is the automorphism of $H_\text{def}/A\cong \overline S$ induced by $\sigma$.

We will use additive notation for the group operation in $A$, multiplicative for $\overline S$, and $\odot_H$ for the group operation in $H_\text{def}$.

Notice that the restriction of $\sigma$ to $A$ is definable: $A^0=\R^n \times \SO_2(\R)^k$, $\R^n$ is the maximal torsion free subgroup and $\SO_2(\R)^k$ is the maximal compact, so the restriction of $\sigma$ to $A^0$ splits and it is the direct product of an element of $GL_n(\mathbb R)$ and an automorphism of a compact matrix algebraic group which must be algebraic. The map $\tau$ is an isomorphism of semisimple connected linear groups, so its graph is a semisimple connected matrix group and therefore definable (\ref{Lie-done-simple}).

\medskip

We now follow \cite{We71}. We will use some of the notation there so that the interested reader can check the computations.

Let $\gamma_1:\overline S\rightarrow A$ be defined by $\sigma(e_A,x)=(\gamma_1(x), \tau(x))$ for all $x$ (so that $\sigma(a,x)=(\sigma_A(x)+\gamma_1(x), \tau(x))$ for all $x\in \overline S$ and $a\in A$). Let $\gamma$ be such that $\gamma(\tau(x))=\gamma_1(x)$ so that $\sigma(a,x)=(\sigma_A(a)+\gamma(\tau(x)), \tau(x))$.

Applying $\sigma$ to the identity $(a,x)\odot_H (b,y)=(a+b+F(x,y),xy)$, we get that
\begin{equation}\label{K}
\gamma(xy)-\gamma(x)-\gamma(y)=F(x,y)-\sigma_A \left(F\left(\tau^{-1}\left(x\right), \tau^{-1}\left(y\right)\right)\right).
\end{equation}

Let $K(x,y)$ be the definable 2-cocycle $F(x,y)-\sigma_A \left(F\left(\tau^{-1}\left(x\right), \tau^{-1}\left(y\right)\right)\right)$.
$K(x,y)$ is a 2-cocycle in $C^2(\overline S, A)$ which by (\ref{K}), is a co-boundary (witnessed by $\gamma$), so that the group $H_K$ defined over $A\times \overline S$ by
\[
(a,x)\odot_K (b,y)=(a+b+K(x,y),xy)
\]
is isomorphic (as a Lie group) to the direct product $A\times \overline S$. Notice that because there are no non-trivial group morphisms from $\overline S$ to $A$, the group $S=\{(e_A,x)\colon x\in S_{\text{def}}\}$ is the unique subgroup of $A\times \overline S$ isomorphic to $\overline S$. So $H_K$ also admits a unique subgroup isomorphic to $\overline S$, let
\[
S_K:=\{(\gamma'(x),x)\}
\]
be such a subgroup. We will prove first that $\gamma=\gamma'$, and then that $S_K$ is definable, which will imply $\gamma'(x)$ is definable.

By defintion of $S_K$, $(\gamma'(x),x)\odot_K (\gamma'(y),y)=(\gamma'(xy),xy)$ which by definition of $\odot_K$ holds if and only if $\gamma'(xy)-\gamma'(x)-\gamma'(y)=K(x,y)$. It follows that $(\gamma-\gamma')(xy)=(\gamma-\gamma')(x)+(\gamma-\gamma')(y)$ and $(\gamma-\gamma')$ is a group homomorphism from $\overline S$ to $A$ which implies $(\gamma-\gamma')(x)=e_A$ and $\gamma=\gamma'$.

We will now show that $S_K$ is definable. Since $S_K$ is isomorphic to $\overline S$ which is perfect, and because $H_K$ is a central extension of $S_K$ we know that $S_K=[H_K, H_K]$ and $[H_K,H_K]_n=[S_K,S_K]_n$ for all $n\in\mathbb N$. Also, $\overline S$ is definable so  $[\overline S,\overline S]_k=\overline S$ for some finite $k$ (Fact 5.3 in \cite{BJO}). It follows that
\[[H_K,H_K]_k=[S_K,S_K]_k=S_K.\]

Hence $S_K$ is definable, and so is $\gamma$, $\gamma_1$ ($\tau$ is definable) and therefore $\sigma$.
\end{proof}

\medskip

\begin{claim}\label{moving ZS}
Let $Z=\mathcal Z(G_{\text{def}})$ and $S$ be a Levi subgroup $G_{\text{def}}$, so that $ZS$ is definably isomorphic to $H_{\text{def}}$ (it is the image of any section of $H_{\text{def}}$ in the quotient sending $H_\text{def} \ltimes_{\psi} R_\text{def}$ to $G_\text{def}$).

Then the restriction $\sigma_{ZS}$ of $\sigma$ to $\mathcal Z(G_\text{def})S$ is a definable map in $\mathbb{R}_{\exp}$ from $ZS$ to $G_\text{def}$.
\end{claim}

\begin{proof}
Any two Levi subgroups are conjugate so $\sigma(S)=\gamma_\sigma S \gamma_\sigma^{-1}$ for some $\gamma_\sigma\in G_\text{def}$.

Since $\sigma$ must fix the center $Z:=\mathcal{Z}(G_\text{def})$ we know that
\[\sigma(ZS)=Z\sigma (S)=Z\gamma_\sigma S \gamma_\sigma^{-1}=\gamma_\sigma (ZS) \gamma_\sigma^{-1},\]
and $\eta_\sigma:x\mapsto \gamma_\sigma^{-1}\sigma(x)\gamma_\sigma$ is an automorphism of $\text{Aut}(ZS)$. This induces (modulo the definable isomorphism between $ZS$ and $H_\text{def}$) an automorphism of $H_\text{def}$ which by Claim \ref{automorphism of H} is definable.

So $\eta_\sigma$ is definable and so is $\sigma$ (conjugation in a definable group is always definable).
\end{proof}

\medskip

We know that $G_\text{def}=\big( \mathcal{Z}\left(G_\text{def}\big)S\right)\mathcal{R}(G_\text{def})$. So for any element $g\in G$ we have $g=sr$ with $s\in \mathcal{Z}\left(G_\text{def}\right)S$ and $r\in \mathcal{R}(G_\text{def})$, and both $\mathcal{Z}\left(G_\text{def}\right)S$ and $\mathcal{R}(G_\text{def})$ are definable subgroups of $G_\text{def}$ by construction.

Therefore $\sigma(g) = \sigma(sr)=\sigma(s)\sigma(r)=\sigma_{ZS}(s)\sigma_{R}(r)$, and $\sigma$ is definable.
\end{proof}

\section{Further comments}
Theorem \ref{definable automorphisms} implies that every group $G$ satisfying the conditions of Theorem \ref{general} has a definable copy such that every automorphism of $G$ is definable. Let $G$ be a Lie group with finitely many connected components. The following fact is well-known, but we could not find it in the literature (except for a proof explained in  https://mathoverflow.net/questions/150949/)

\begin{fact}
Let $G$ be a Lie group with finitely many connected components. Then $G = F G^0$, for some finite $F < G$.
\end{fact}

By Lemma \ref{centralextension} we may assume that $G$ itself is a finite definable extension of some $G^0$ satisfying the conditions of Remark \ref{the right form}, so that by Theorem \ref{definable automorphisms} every Lie automorphism of $G^0$ is definable.

Now, $F$ and $\sigma\mid _F$ are both definable (they are finite) and for any $a\in F, x\in G^0$ and $\sigma\in Aut(G)$ we have
\[
\sigma(ax)=\sigma(a)\sigma(x)
\]
which is definable.

\bigskip

Notice also that in the proof of Claim \ref{automorphism of H} we show in fact that if $H$ is a central extension of a definable semisimple group $S$ by an abelian semialgebraic group $A$, then any automorphism of $H$ is semialgebraic.

\bigskip

In our groups and maps definable in  $\mathbb{R}_{\exp}$, the exponential function is needed for definability of automorphisms of upper triangular groups. However, for nilpotent upper triangular Lie algebras the exponential map is polynomial, and therefore semialgebraic. It seems that this is the only use of the exponential function in our analysis, so that if we assume that a group $G$ has finitely many connected components, its center has finitely many connected components and the solvable radical is nilpotent by compact, then it has a definable copy which is semialgebraic. In other words, if $G$ has a definable copy $G_\text{def}$ and $\mathcal{N}(G_\text{def})$ is nilpotent,  then we can find a semialgebraic copy.


\end{document}